\documentclass[10pt,reqno]{amsart}
\usepackage{bbm}
\usepackage{mathrsfs}
\usepackage{amsfonts} 
\usepackage[dvipsnames,usenames]{color}
\usepackage{graphicx,latexsym,bm,amsmath,amssymb,verbatim,multicol,lscape}
\usepackage{enumerate}

\newtheorem{thm}{Theorem} [section]
\newtheorem{lem}{Lemma}[section]
\newtheorem{pro}{Proposition}[section]
\newtheorem{cor}{Corollary}[section]

\theoremstyle{definition}
\newtheorem{defn}{Definition}[section]
\theoremstyle{remark}

\numberwithin{equation}{section}

\allowdisplaybreaks

\begin{document}
\title[The Normal Form of Smith's Matrices]
{The Normal Form of Smith's Matrices}
\author[W.Z. Lei]{Wenzhong Lei}
\address{Mathematical College, Sichuan University,
Chengdu 610064, P.R. China}
\email{lwzh1729@163.com}

\author[H. Zhang]{Han Zhang}
\email{1901110026@pku.edu.cn}


\begin{abstract}
For any integers $x$ and $y$, let $(x,y)$ and $[x,y]$ stand for the greatest common divisor and the least common
multiple of $x$ and $y$, respectively. We denote by $|T|$ the number of elements of a finite set $T$. Let $a,b$ and $n$
be positive integers and let $S=\{x_1,...,x_n\}$ be a set of $n$ distinct positive integers. Let $(f((x_i,x_j)))$ (abbreviated by $f(S)$) and $(f([x_i,x_j]))$
 (abbreviated by $(f([S]))$) stand for the $n\times n$ matrices whose $(i,j)-$entry is $(f((x_i,x_j)))$ and $(f([x_i,x_j]))$
 respectively. In 1989, Beslin and Ligh
gave a description of the lower triangular decomposition of $((x_i,x_j))$. In 1992, Bourque and Ligh
showed that if $S$ is factor closed (i.e., S contains all positive divisors of any element of S), then the GCD matrix
$((x_i,x_j))$ divides the LCM matrix $([x_i,x_j])$ (written as $((x_i,x_j))|([x_i,x_j])$) in the ring $M_n(\mathbb{Z})$
of $n\times n$ matrices over the integers.
 In this paper, we will show the diagonalization of $((x_i,x_j))$ and its applications. Our main new contributions are Theorems 4.1 and 4.2, which extend previous results to gcd-closed sets satisfying condition $\mathcal{G}$.
\end{abstract}
\maketitle

\section{Introduction}
For arbitrary positive integers $x$ and $y$, we denote by $(x,y)$ the greatest common divisor of $x$ and $y$ and $[x,y]$
their least common multiple. Throughout, let $a,b$ and $n$ be positive integers and let $f$ be arithmetic function. Let
$S=\{x_1,...,x_n\}$ be a set of $n$ distinct positive integers. Let $(f((x_i,x_j)))$ (abbreviated by $f(S)$) and $(f([x_i,x_j]))$
 (abbreviated by $(f([S]))$) stand for the $n\times n$ matrices whose $(i,j)-$entry is $(f((x_i,x_j)))$ and $(f([x_i,x_j]))$
 respectively. The set $S$ is called \textbf{factor closed}(FC) if all positive divisors of any element of $S$ are in $S$, and is
 called \textbf{gcd closed} if $(x_i,x_j)\in S$ for all integers $i$ and $j$ with $1\le i,j \le n$. An FC set is clearly gcd closed
 but the converse is not true. For example, if $a>1$ is an integer, then the set $S=\{a,2a,...,na\}$ is gcd closed but not factor closed.
 In 1875, Smith \cite{[S]} published his renowned result stating that if $S$ is an FC set, then $det(f(S))=\prod_{k=1}^n(f*\mu)(x_k)$, where
 $\mu$ is the M\H{o}bius function and $f*\mu$ is the Dirichlet convolution of $f$ and $\mu$. In 1989, Beslin and Ligh \cite{[BL]}
gave a description of the lower triangular decomposition of $((x_i,x_j))$. Similar factorization results have been obtained in the works of Bourque and Ligh \cite{[BL2],[BL3]}. Inspired by this, we have the following result, which is closely related to known diagonalization theorems for GCD matrices (see, e.g., \cite{Ilmonen, Altinisik, Ovall, Bhat}).

\begin{thm}\label{thm1.1}\cite{[BL2],[BL3],Ilmonen}
Assume that $S=\{x_1,...,x_n\}$ is a factor closed set. Let $B=(b_{ij})$, where
$$
b_{ij}=
\begin{cases}
\mu(\frac{x_j}{x_i}), & \text{if } x_i|x_j;\\
0, & \text{otherwise}.
\end{cases}
$$
Then for an arbitrary function $f$, we have
$$B^T(f((x_i,x_j)))B=diag(f*\mu(x_1),f*\mu(x_2),...,f*\mu(x_n)).$$
\end{thm}

If $(f([x_i,x_j]))=([x_i,x_j]^a)$, we can give a special normal form.

\begin{thm}\label{thm1.2}
Assume that $S=\{x_1,...,x_n\}$ is a factor closed set. Let $[B^a]=(b^a_{ij})_{1\le i,j\le n}$, where
$$
b^a_{ij}=
\begin{cases}
\mu(\frac{x_j}{x_i})(\frac{x_j}{x_i})^a, & \text{if } x_i|x_j;\\
0, & \text{otherwise}.
\end{cases}
$$
Then we have
$$[B^a]^T([x_i,x_j]^a)[B_a]=diag(1,\prod_{p|x_2}(p^{a\nu_p(x_2)}-p^{a(\nu_p(x_2)+1)}),...,\prod_{p|x_n}(p^{a\nu_p(x_n)}-p^{a(\nu_p(x_n)+1)})),$$
 where, without loss of generality, we assume that $x_1=1$.
\end{thm}

In 2008, Hong \cite{[H]} proved that $((x_i,x_j)^a)|((x_i,x_j)^b)$, $((x_i,x_j)^a)|([x_i,x_j]^b)$, and $\quad$
$([x_i,x_j]^a)|([x_i,x_j]^b)$
hold in the ring $M_n(\mathbb{Z})$ when $a|b$ and $S$ is a divisor chain (namely, there is a permutation $\sigma$ of order $n$ such that
$x_{\sigma(1)}|...|x_{\sigma(n)}$). We can establish the result again and extend it by Theorem \ref{thm1.1} and Theorem \ref{thm1.2}, and the result with different method is in \cite{[H3]}.

Moreover, when $f$ is (completely) multiplicative, we have the general case of  Theorem \ref{thm1.1} and Theorem \ref{thm1.2}, which will be given in Section 3.

Furthermore, if $S$ is gcd closed, the problem may be much more complicated. However, we can add some conditions to obtain a set similar to an FC set, and inspired by Theorem \ref{thm1.1} and Theorem \ref{thm1.2}, we will consider the similar divisibility problems by using imitative M\H{o}bius function, which is defined in Section 4.

This paper is organized as follows. In Section 2, we prove Theorem \ref{thm1.1} and Theorem \ref{thm1.2}. In Section 3, we extend
the result of \cite{[H]} by Theorem \ref{thm1.1} and Theorem \ref{thm1.2}. In Section 4, we consider a special generalization of FC
 set, which is a gcd closed set.

\section{The Proof of Theorem \ref{thm1.1} and Theorem \ref{thm1.2}}

In this section, first we prove a lemma and by using the lemma, we can finish the proof of  Theorem \ref{thm1.1} and Theorem \ref{thm1.2}.
 we denote by $\nu_p(x)$ \textbf{the p-adic valuation of} $x$, i.e. the largest nonegative integer $r$ such that $p^r$ divides $x$ for
  any integer $x$. One writes $p^l||x$ if $l=\nu_p(x)$.

\begin{lem}\label{lem2.1}
For $m,n\in\mathbb{Z}_+$,
$$
\sum_{d|n}\mu(\frac{n}{d})f((d,m))=
\begin{cases}
f*\mu(n), & \text{if } n|m,\\
0, & \text{if } n\nmid m.
\end{cases}
$$
\end{lem}

\begin{proof}
First, when $n|m$, we have $(d,m)=d$, and then
$$\sum_{d|n}\mu(\frac{n}{d})f((d,m))=\sum_{d|n}\mu(\frac{n}{d})f(d)=f*\mu(n).$$

Now if $n\nmid m$, then there exists a prime $p$ such that $p|n$ and $p^{\nu_p(n)}\nmid m$. Moreover, for $p^{\nu_p(n)}||d$,
we have $(d,m)=(\frac{d}{p},m)$, and
$$\mu(\frac{n}{d})f((d,m))+\mu(\frac{np}{d})f((\frac{d}{p},m))=\mu(\frac{n}{d})f((d,m))(1+\mu(p))=0,$$
where clearly we have $(p,\frac{n}{d})=1$.

Therefore,
$$
\sum_{d|n}\mu(\frac{n}{d})f((d,m))=
\sum_{d|n,p^{\nu_p(n)}||d}\left(\mu(\frac{n}{d})f((d,m))+\mu(\frac{np}{d})f((\frac{d}{p},m))\right)=0.
$$

Thus, Lemma \ref{lem2.1} is proved.
\end{proof}

The next corollary is natural.

\begin{cor}\label{cor2.1}
When $f(x)=x^k$, then
$$
\sum_{d|n}\mu(\frac{n}{d})f((d,m))=
\begin{cases}
\prod_{p|n}(p^{k\nu_p(n)}-p^{k(\nu_p(n)-1)}), & \text{if } n|m,\\
0, & \text{if } n\nmid m,
\end{cases}
$$
where $p$ is a prime.
\end{cor}

\begin{proof}
Write $n=p_1^{\nu_{p_1}(n)}p_2^{\nu_{p_2}(n)}...p_e^{\nu_{p_e}(n)}$.

By Lemma \ref{lem2.1}, we only need to compute
$$\sum_{d|n}\mu(\frac{n}{d})d^k=n^k+\sum_{i=1}^e\mu(p_i)(\frac{n}{p_i})^k+...+\mu(\prod_{i=1}^ep_i)(\frac{n}{\prod_{i=1}^ep_i})^k$$
$$=\prod_{i=1}^e(p_i^{k\nu_{p_i}(n)}+\mu(p_i)p_i^{k(\nu_{p_i}(n)-1)})=\prod_{i=1}^e(p_i^{k\nu_{p_i}(n)}-p_i^{k(\nu_{p_i}(n)-1)}),$$
and the corollary is proved.
\end{proof}

Now we can prove Theorem \ref{thm1.1} and Theorem \ref{thm1.2}.

\begin{proof}
First, consider
$$(f((x_i,x_j)))B=(a_{ij}),$$
where
$$a_{ij}=\sum_{x_k|x_j} f((x_i,x_k))\mu(\frac{x_j}{x_k}).$$
By Lemma \ref{lem2.1},
$$
\sum_{x_k|x_j} f((x_i,x_k))\mu(\frac{x_j}{x_k})=
\begin{cases}
f*\mu(x_j), & \text{if } x_j|x_i,\\
0, & \text{otherwise}.
\end{cases}
$$

Now we have
$$B^T(f((x_i,x_j)))B=(a'_{ij}),$$
where
$$a'_{ij}=\sum_{x_j|x_k|x_i}\mu(\frac{x_i}{x_k})a_{kj}=\sum_{x_j|x_k|x_i}\mu(\frac{x_i}{x_k})(f*\mu(x_j))$$
$$
=\begin{cases}
f*\mu(x_j), & \text{if } i=j,\\
0, & \text{otherwise}.
\end{cases}
$$

Thus,
$$B^T(f((x_i,x_j)))B=diag(f*\mu(x_1),f*\mu(x_2),...,f*\mu(x_n)).$$

Next, we calculate $[B^a]^T([x_i,x_j]^a)[B_a]$.

Let $([x_i,x_j]^a)[B_a]=(h^a_{ij})$, where
$$h^a_{ij}=\sum_{x_k|x_j}[x_i,x_k]^a\mu(\frac{x_j}{x_k})(\frac{x_j}{x_k})^a
=\sum_{x_k|x_j}\left(\frac{x_ix_k}{(x_i,x_k)}\right)^a\mu(\frac{x_j}{x_k})(\frac{x_j}{x_k})^a$$
$$=(x_ix_j)^a\sum_{x_k|x_j}\mu(\frac{x_j}{x_k})(x_i,x_k)^{-a}.$$
By Corollary \ref{cor2.1},
$$
\sum_{x_k|x_j}\mu(\frac{x_j}{x_k})(x_i,x_k)^{-a}
=\begin{cases}
\prod_{p|x_j}(p^{-a\nu_p(x_j)}-p^{-a(\nu_p(x_j)-1)}), & \text{if } x_j|x_i,\\
0, & \text{otherwise}.
\end{cases}
$$

Therefore,
$$
h^a_{ij}=\begin{cases}
(x_ix_j)^a\prod_{p|x_j}(p^{-a\nu_p(x_j)}-p^{-a(\nu_p(x_j)-1)})=x_i^a\prod_{p|x_j}(1-p^a), & \text{if } x_j|x_i,\\
0, & \text{otherwise}.
\end{cases}
$$

Now we have
$$[B_a]^T([x_i,x_j]^a)[B_a]=(h'^a_{ij}),$$
where
$$
h'^a_{ij}=\sum_{x_j|x_k|x_i}\mu(\frac{x_i}{x_k})(\frac{x_i}{x_k})^ah^a_{ij}=\sum_{x_j|x_k|x_i}\mu(\frac{x_i}{x_k})(\frac{x_i}{x_k})^ax_k^a\prod_{p|x_j}(1-p^a)
$$$$=x_i^a\prod_{p|x_j}(1-p^a)\sum_{x_j|x_k|x_i}\mu(\frac{x_i}{x_k})$$
$$
=\begin{cases}
x_j^a\prod_{p|x_j}(1-p^a)=\prod_{p|x_j}(p^{a\nu_p(x_j)}-p^{a(\nu_p(x_j)+1)})), & \text{if } i=j,\\
0, & \text{otherwise}.
\end{cases}
$$

Hence,
$$[B^a]^T([x_i,x_j]^a)[B_a]=diag(1,\prod_{p|x_2}(p^{a\nu_p(x_2)}-p^{a(\nu_p(x_2)+1)}),...,\prod_{p|x_n}(p^{a\nu_p(x_n)}-p^{a(\nu_p(x_n)+1)})).$$

This completes the proof of Theorem \ref{thm1.1} and Theorem \ref{thm1.2}.
\end{proof}

\section{The Application in Divisibility Properties}
In this section, we use Theorem \ref{thm1.1} and Theorem \ref{thm1.2} to consider the divisibility properties
of power GCD matrices and power LCM matrices. For $A,B\in M_n(\mathbb{Z})$, we say that $A|B$, if there exists a matrix $C\in M_n(\mathbb{Z})$ such that
$B=AC$ or $B=CA$. We can find the same result with different method in \cite{[H3]}.

\begin{thm}\label{thm3.1}
Let $S=\{x_1,...,x_n\}$ be a factor closed set, and let $a,b\in\mathbb{Z}_+$ with $a|b$. Then we have
$((x_i,x_j)^a)|((x_i,x_j)^b)$, $((x_i,x_j)^a)|([x_i,x_j]^b)$ and $([x_i,x_j]^a)|([x_i,x_j]^b)$.
\end{thm}

By Theorem \ref{thm1.1} and Corollary \ref{cor2.1}, we have

\begin{cor}\label{cor3.1}
Assume that $S=\{x_1,...,x_n\}$ is a factor closed set, and define $B\in M_n(\mathbb{Z})$
as in Theorem \ref{thm1.1}. Then
$$B^T((x_i,x_j)^a)B=diag(1,\prod_{p|x_2}(p^{a\nu_p(x_2)}-p^{a(\nu_p(x_2)-1)}),...,\prod_{p|x_n}(p^{a\nu_p(x_n)}-p^{a(\nu_p(x_n)-1)})).$$
\end{cor}

We need to compute $B^{-1}$ and $[B_a]^{-1}$ first.

\begin{lem}\label{lem3.1}
The inverse of $B$ is $B^{-1}=(u_{ij})$, where
$$
u_{ij}=\begin{cases}
1, & \text{if } x_i|x_j,\\
0, & \text{otherwise}.
\end{cases}
$$
The inverse of $[B_a]$ is $[B_a]^{-1}=(u_{a,ij})$, where
$$u_{a,ij}=\begin{cases}
(\frac{x_j}{x_i})^a, & \text{if } x_i|x_j,\\
0, & \text{otherwise}.
\end{cases}$$
\end{lem}

\begin{proof}
Let $BB^{-1}=(c_{ij})$, where
$$
c_{ij}=\sum_{x_i|x_k|x_j}\mu(\frac{x_k}{x_i})=\begin{cases}
1, & \text{if } i=j,\\
0, & \text{otherwise}.
\end{cases}
$$

Let $[B_a][B_a]^{-1}=(c^a_{ij})$, where
$$c^a_{ij}=\sum_{x_i|x_k|x_j}\mu(\frac{x_k}{x_i})(\frac{x_k}{x_i})^a(\frac{x_j}{x_k})^a
=(\frac{x_j}{x_i})^a\sum_{d|\frac{x_j}{x_i}}\mu(d)$$
$$
=\begin{cases}
1, & \text{if } i=j,\\
0, & \text{otherwise}.
\end{cases}
$$

Thus, the proof of Lemma \ref{lem3.1} is complete.
\end{proof}

Now we can prove Theorem \ref{thm3.1}. Without loss of generality, assume that $a\ne b$.

\begin{proof}
By Corollary \ref{cor3.1} and Lemma \ref{lem3.1},
\begin{align*}
&((x_i,x_j)^a)^{-1}((x_i,x_j)^b) \\
&=((B^T)^{-1}diag(1,\prod_{p|x_2}(p^{a\nu_p(x_2)}-p^{a(\nu_p(x_2)-1)}),...,\prod_{p|x_n}(p^{a\nu_p(x_n)}-p^{a(\nu_p(x_n)-1)}))B^{-1})^{-1}\\
&(B^T)^{-1}diag(1,\prod_{p|x_2}(p^{b\nu_p(x_2)}-p^{b(\nu_p(x_2)-1)}),...,\prod_{p|x_n}(p^{b\nu_p(x_n)}-p^{b(\nu_p(x_n)-1)}))B^{-1}\\
&=Bdiag(1,\prod_{p|x_2}\frac{1}{p^{a\nu_p(x_2)}-p^{a(\nu_p(x_2)-1)}},...,\prod_{p|x_n}\frac{1}{p^{a\nu_p(x_n)}-p^{a(\nu_p(x_n)-1)}})\\
&diag(1,\prod_{p|x_2}(p^{b\nu_p(x_2)}-p^{b(\nu_p(x_2)-1)}),...,\prod_{p|x_n}(p^{b\nu_p(x_n)}-p^{b(\nu_p(x_n)-1)}))B^{-1}\\
&=Bdiag(1,\prod_{p|x_2}\frac{p^{b\nu_p(x_2)}-p^{b(\nu_p(x_2)-1)}}{p^{a\nu_p(x_2)}-p^{a(\nu_p(x_2)-1)}},...,\prod_{p|x_n}\frac{p^{b\nu_p(x_n)}-p^{b(\nu_p(x_n)-1)}}{p^{a\nu_p(x_n)}-p^{a(\nu_p(x_n)-1)}})B^{-1}.
\end{align*}

Since $B,B^{-1}\in M_n(\mathbb{Z})$, and because $a|b$, we have $\prod_{p|x_i}\frac{p^{b\nu_p(x_i)}-p^{b(\nu_p(x_i)-1)}}{p^{a\nu_p(x_i)}-p^{a(\nu_p(x_i)-1)}}\in\mathbb{Z}$. Therefore,
$diag(1,\prod_{p|x_2}\frac{p^{b\nu_p(x_2)}-p^{b(\nu_p(x_2)-1)}}{p^{a\nu_p(x_2)}-p^{a(\nu_p(x_2)-1)}},...,\prod_{p|x_n}\frac{p^{b\nu_p(x_n)}-p^{b(\nu_p(x_n)-1)}}{p^{a\nu_p(x_n)}-p^{a(\nu_p(x_n)-1)}})\in M_n(\mathbb{Z})$.

Hence, $((x_i,x_j)^a)^{-1}((x_i,x_j)^b)\in M_n(\mathbb{Z})$.

Next, we prove that $((x_i,x_j)^a)|([x_i,x_j]^a)$ and $([x_i,x_j]^a)|([x_i,x_j]^b)$.

By Theorem \ref{thm1.2}, Corollary \ref{cor3.1} and Lemma \ref{lem3.1},
\begin{align*}
&((x_i,x_j)^a)^{-1}([x_i,x_j]^a)=\\
&Bdiag(1,\prod_{p|x_2}\frac{1}{p^{a\nu_p(x_2)}-p^{a(\nu_p(x_2)-1)}},...,\prod_{p|x_n}\frac{1}{p^{a\nu_p(x_n)}-p^{a(\nu_p(x_n)-1)}})B^T([B_a]^T)^{-1}\\
&diag(1,\prod_{p|x_2}(p^{a\nu_p(x_2)}-p^{a(\nu_p(x_2)+1)}),...,\prod_{p|x_n}(p^{a\nu_p(x_n)}-p^{a(\nu_p(x_n)+1)}))[B_a]^{-1}.
\end{align*}

Since $B,[B_a]^{-1}\in M_n(\mathbb{Z})$, we only need to compute
\begin{align*}
H_a:=&diag(1,\prod_{p|x_2}\frac{1}{p^{a\nu_p(x_2)}-p^{a(\nu_p(x_2)-1)}},...,\prod_{p|x_n}\frac{1}{p^{a\nu_p(x_n)}-p^{a(\nu_p(x_n)-1)}})B^T\\
&([B_a]^T)^{-1}diag(1,\prod_{p|x_2}(p^{a\nu_p(x_2)}-p^{a(\nu_p(x_2)+1)}),...,\prod_{p|x_n}(p^{a\nu_p(x_n)}-p^{a(\nu_p(x_n)+1)})).
\end{align*}

By Lemma \ref{lem3.1} and Corollary \ref{cor2.1}, $B^T([B_a]^T)^{-1}=(g_{ij})$, where
$$g_{ij}=\sum_{x_j|x_k|x_i}\mu(\frac{x_i}{x_k})(\frac{x_k}{x_j})^a=\sum_{d|\frac{x_i}{x_j}}\mu(d)(\frac{x_i}{dx_j})^a
=\prod_{p|\frac{x_i}{x_j}}(p^{a\nu_p(\frac{x_i}{x_j})}-p^{a(\nu_p(\frac{x_i}{x_j})-1)}),$$
and $g_{ij}=0$ when $x_j\nmid x_i$.

Thus the $(i,j)$-entry of $H_a$ is
\begin{align*}
g'_{ij}:=&\left(\prod_{p|x_i}\frac{1}{p^{a\nu_p(x_i)}-p^{a(\nu_p(x_i)-1)}}\right)
g_{ij}\left(\prod_{p|x_j}(p^{a\nu_p(x_j)}-p^{a(\nu_p(x_j)+1)})\right)\\
=&\left(\prod_{p|x_i}\frac{1}{p^{a\nu_p(x_i)}-p^{a(\nu_p(x_i)-1)}}\right)\left(\prod_{p|\frac{x_i}{x_j}}(p^{a\nu_p(\frac{x_i}{x_j})}-p^{a(\nu_p(\frac{x_i}{x_j})-1)})\right)\\
&\left(\prod_{p|x_j}(p^{a\nu_p(x_j)}-p^{a(\nu_p(x_j)+1)})\right)\\
=&\frac{\prod_{p|x_j}p^a\prod_{p|\frac{x_i}{x_j}}(p^a-1)\prod_{p|x_j}(1-p^a)}{\prod_{p|x_i}(p^a-1)},
\end{align*}
where we use the fact that when $\nu_p(\frac{x_i}{x_j})>0$, $p^{a(\nu_p(x_i)-1)}=p^{a\nu_p(x_j)+a(\nu_p(\frac{x_i}{x_j})-1)}$. Moreover,
$$\prod_{p|\frac{x_i}{x_j}}(p^a-1)\prod_{p|x_j}(p^a-1)=\prod_{p|x_i}(p^a-1)\prod_{p|(x_j,\frac{x_i}{x_j})}(p^a-1),$$
so $g'_{ij}\in \mathbb{Z}$.

Therefore, $H_a\in M_n(\mathbb{Z})$, and hence $((x_i,x_j)^a)^{-1}([x_i,x_j]^a)\in M_n(\mathbb{Z})$.

Finally, by Theorem \ref{thm1.2} and Lemma \ref{lem3.1},
\begin{align*}
&([x_i,x_j]^a)^{-1}([x_i,x_j]^b)=\\
&[B^a]^Tdiag\{1,\frac{1}{\prod_{p|x_2}(p^{a\nu_p(x_2)}-p^{a(\nu_p(x_2)+1)})},...,\frac{1}{\prod_{p|x_n}(p^{a\nu_p(x_n)}-p^{a(\nu_p(x_n)+1)})}\}[B^a]\\
&([B^b]^T)^{-1}diag\{1,\prod_{p|x_2}(p^{b\nu_p(x_2)}-p^{b(\nu_p(x_2)+1)}),...,\prod_{p|x_n}(p^{b\nu_p(x_n)}-p^{b(\nu_p(x_n)+1)})\}[B^b]^{-1}.
\end{align*}

Since $[B_a]^T,[B^b]^{-1}\in M_n(\mathbb{Z})$, we only need to compute
\begin{align*}
H'_a:=&diag\{1,\frac{1}{\prod_{p|x_2}(p^{a\nu_p(x_2)}-p^{a(\nu_p(x_2)+1)})},...,\frac{1}{\prod_{p|x_n}(p^{a\nu_p(x_n)}-p^{a(\nu_p(x_n)+1)})}\}[B^a]\\
&([B^b]^T)^{-1}diag\{1,\prod_{p|x_2}(p^{b\nu_p(x_2)}-p^{b(\nu_p(x_2)+1)}),...,\prod_{p|x_n}(p^{b\nu_p(x_n)}-p^{b(\nu_p(x_n)+1)})\}.
\end{align*}

By Lemma \ref{lem3.1} and Corollary \ref{cor2.1}, $[B_a]([B_b]^T)^{-1}=(h_{ij})$, where
\begin{align*}
h_{ij}&=\sum_{x_j|x_k|x_i}\mu(\frac{x_i}{x_k})(\frac{x_i}{x_k})^a(\frac{x_k}{x_j})^b\\
&=(\frac{x_i}{x_j})^a\sum_{d|\frac{x_i}{x_j}}\mu(\frac{x_i}{dx_j})d^{b-a}\\
&=(\frac{x_i}{x_j})^a\prod_{p|\frac{x_i}{x_j}}(p^{(b-a)\nu_p(\frac{x_i}{x_j})}-p^{(b-a)(\nu_p(\frac{x_i}{x_j})-1)}).
\end{align*}
Thus the $(i,j)$-entry of $H'_a$ is
\begin{align*}
h'_{ij}:=&\frac{1}{\prod_{p|x_i}(p^{a\nu_p(x_i)}-p^{a(\nu_p(x_i)+1)})}h_{ij}\prod_{p|x_j}(p^{b\nu_p(x_j)}-p^{b(\nu_p(x_j)+1)})\\
=&\frac{\prod_{p|x_j}(p^{b\nu_p(x_j)}-p^{b(\nu_p(x_j)+1)})}{\prod_{p|x_i}(p^{a\nu_p(x_i)}-p^{a(\nu_p(x_i)+1)})}(\frac{x_i}{x_j})^a\prod_{p|\frac{x_i}{x_j}}(p^{(b-a)\nu_p(\frac{x_i}{x_j})}-p^{(b-a)(\nu_p(\frac{x_i}{x_j})-1)})\\
=&\frac{[(\frac{x_i}{x_j})^b\prod_{p|\frac{x_i}{x_j}}(1-\frac{1}{p^{b-a}})][x_j^b\prod_{p|x_j}(1-p^b)]}{x_i^a\prod_{p|x_i}(1-p^a)}\\
=&\frac{(x_i)^{b-a}\prod_{p|\frac{x_i}{x_j}}({p^{b-a}-1})\prod_{p|x_j}(1-p^b)}{\prod_{p|\frac{x_i}{x_j}}p^{b-a}\prod_{p|x_i}(1-p^a)}.
\end{align*}

Note that
$$\prod_{p|\frac{x_i}{x_j}}p^{b-a}\mid (x_i)^{b-a},\quad \prod_{p|x_i}(1-p^a)\mid \prod_{p|\frac{x_i}{x_j}}({p^{b-a}-1})\prod_{p|x_j}(1-p^b),$$
and hence $H'_a\in M_n(\mathbb{Z})$, which implies $([x_i,x_j]^a)^{-1}([x_i,x_j]^b)\in M_n(\mathbb{Z})$.

This completes the proof of Theorem \ref{thm3.1}.
\end{proof}

Inspired by Theorem \ref{thm3.1}, we may ask whether the matrix $(f^a((x_i,x_j)))$ with
 (completely) multiplicative function $f$ has similar divisibility properties. In fact, we have

\begin{thm}\label{thm3.2}
Let $S$, $a,b$ be as in Theorem \ref{thm3.1}.

If $f$ is multiplicative, then $(f^a((x_i,x_j)))|(f^b((x_i,x_j)))$.

If $f$ is completely multiplicative, then
$$(f^a((x_i,x_j)))|(f^b((x_i,x_j))),\quad (f^a((x_i,x_j)))|(f^b([x_i,x_j])) \text{ and } (f^a([x_i,x_j]))|(f^b([x_i,x_j])).$$
\end{thm}

Similar to Corollary \ref{cor2.1}, by Lemma \ref{lem2.1}, we have

\begin{cor}\label{cor3.2}
If $f$ is multiplicative, then
$$
\sum_{d|n}\mu(\frac{n}{d})f^k((d,m))=
\begin{cases}
\prod_{p|n}(f^k(p^{\nu_p(n)})-f^k(p^{\nu_p(n)-1})), & \text{if } n|m,\\
0, & \text{if } n\nmid m,
\end{cases}
$$
and if $f$ is completely multiplicative, then
$$
\sum_{d|n}\mu(\frac{n}{d})f^k((d,m))=
\begin{cases}
\prod_{p|n}(f(p)^{k\nu_p(n)}-f(p)^{k(\nu_p(n)-1)}), & \text{if } n|m,\\
0, & \text{if } n\nmid m,
\end{cases}
$$
where $p$ is a prime.
\end{cor}

\begin{proof}
Similar to the proof of Corollary \ref{cor2.1}, we only need to compute
$$\sum_{d|n}\mu(\frac{n}{d})f^k(d)=f^k(n)+\sum_{i=1}^e\mu(p_i)f^k(\frac{n}{p_i})+...+\mu(\prod_{i=1}^ep_i)f^k(\frac{n}{\prod_{i=1}^ep_i})$$
$$=\prod_{i=1}^e(f^k(p_i^{\nu_{p_i}(n)})+\mu(p_i)f^k(p_i^{\nu_{p_i}(n)-1}))=\prod_{i=1}^e(f^k(p_i^{\nu_{p_i}(n)})-f^k(p_i^{\nu_{p_i}(n)-1})).$$
When $f$ is completely multiplicative, $f(p^\nu)=f^\nu(p)$, so
$$\prod_{i=1}^e(f^k(p_i^{\nu_{p_i}(n)})-f^k(p_i^{\nu_{p_i}(n)-1}))=\prod_{i=1}^e(f^{k\nu_{p_i}(n)}(p_i)-f^{k(\nu_{p_i}(n)-1)}(p_i)).$$

This completes the proof.
\end{proof}

By Corollary \ref{cor3.2}, similar to Theorem \ref{thm1.2} and Corollary \ref{cor3.1}, we have

\begin{cor}\label{cor3.3}
Let $S$ be as in Theorem \ref{thm3.1}. If $f$ is multiplicative, then
\begin{align*}
&B^T(f^a(x_i,x_j))B=\\
&diag(1,\prod_{p|x_2}(f^a(p^{\nu_p(x_2)})-f^a(p^{\nu_p(x_2)-1})),...,\prod_{p|x_n}(f^a(p^{\nu_p(x_n)})-f^a(p^{\nu_p(x_n)-1}))).\tag{3.1}
\end{align*}
If $f$ is completely multiplicative, then
\begin{align*}
&B^T(f^a(x_i,x_j))B=\\
&diag(1,\prod_{p|x_2}(f^{a\nu_p(x_2)}(p)-f^{a(\nu_p(x_2)-1)}(p)),...,\prod_{p|x_n}(f^{a\nu_p(x_n)}(p)-f^{a(\nu_p(x_n)-1)}(p))).\tag{3.2}
\end{align*}
Moreover,
\begin{align*}
&[B_{f^a}]^T(f^a[x_i,x_j])[B_{f^a}]=\\
&diag(1,\prod_{p|x_2}(f^{a\nu_p(x_2)}(p)-f^{a(\nu_p(x_2)+1)}(p)),...,\prod_{p|x_n}(f^{a\nu_p(x_n)}(p)-f^{a(\nu_p(x_n)+1)}(p))),\tag{3.3}
\end{align*}
where $[B_{f^a}]=(b^{f^a}_{ij})$ with
$$
b^{f^a}_{ij}=\begin{cases}
\mu(\frac{x_j}{x_i})f((\frac{x_j}{x_i}))^a, & \text{if } x_i|x_j,\\
0, & \text{otherwise}.
\end{cases}
$$
\end{cor}

\begin{proof}
If $f$ is (completely) multiplicative, we use Theorem \ref{thm1.1} and Corollary \ref{cor3.2} to obtain (3.1) and (3.2).

If $f$ is completely multiplicative, we modify the proof of Theorem \ref{thm1.2} by replacing all integers $x_i,d,p$
with $f(x_i),f(d),f(p)$ to prove (3.3).
\end{proof}

Similarly, as in Lemma \ref{lem3.1}, we find that $[B_{f^a}]^{-1}=(b'^{f^a}_{ij})_{1\le i,j \le n}$, where
$$
b'^{f^a}_{ij}=
\begin{cases}
f(\frac{x_i}{x_j})^a, & \text{if } x_j|x_i,\\
0, & \text{otherwise}.
\end{cases}
$$

Now we can complete the proof of Theorem \ref{thm3.2}.

\begin{proof}
Similar to Theorem \ref{thm3.1}, if $f$ is (completely) multiplicative, then
\begin{align*}
&(f^a(x_i,x_j))^{-1}(f^b(x_i,x_j))=\\
&Bdiag(1,\prod_{p|x_2}\frac{f^b(p^{\nu_p(x_2)})-f^b(p^{\nu_p(x_2)-1})}{f^a(p^{\nu_p(x_2)})-f^a(p^{\nu_p(x_2)-1})},...,\prod_{p|x_n}\frac{f^b(p^{\nu_p(x_n)})-f^b(p^{\nu_p(x_n)-1})}{f^a(p^{\nu_p(x_n)})-f^a(p^{\nu_p(x_n)-1})})B^{-1}.
\end{align*}

If $f$ is completely multiplicative, then
$$(f^a(x_i,x_j))^{-1}(f^a[x_i,x_j])=BH_{f^a}[B_{f^a}]^{-1},$$
where the $(i,j)$-entry of $H_{f^a}$ is
$$\frac{\prod_{p|x_j}f^a(p)\prod_{p|\frac{x_i}{x_j}}(f^a(p)-1)\prod_{p|x_j}(1-f^a(p))}{\prod_{p|x_i}(f^a(p)-1)}\in\mathbb{Z},$$
and
$$(f^a[x_i,x_j])^{-1}(f^b[x_i,x_j])=[B_{f^a}]H_{f^a,f^b}[B_{f^b}]^{-1},$$
where the $(i,j)$-entry of $H_{f^a,f^b}$ is
$$\frac{f^{b-a}(x_i)\prod_{p|\frac{x_i}{x_j}}(f^{b-a}(p)-1)\prod_{p|x_j}(1-f^b(p))}
{\prod_{p|\frac{x_i}{x_j}}f^{b-a}(p)\prod_{p|x_i}(1-f^a(p))}.$$

By a similar argument as in the proof that $([x_i,x_j]^a)^{-1}([x_i,x_j]^b)\in M_n(\mathbb{Z})$, we have $H_{f^a,f^b}\in M_n(\mathbb{Z})$.

This completes the proof of Theorem \ref{thm3.2}.
\end{proof}

\section{A Generalization of the FC Set}
If $S$ is gcd closed, the problem may be much more complicated. However, we can add some conditions to obtain a set similar to an FC set.

We denote by $G_S(x)$ the set of all the greatest-type divisors of $x$ in $S$ (see \cite{[H2]}). For any set $R$ of positive integers
and for any $x\in R$, we define $g_R(x)$ to be the number of greatest-type divisors of $x$ in $R$, i.e., $g_R(x):=|G_R(x)|$.
For brevity, we write $g(x)$ for $g_R(x)$.

\begin{defn}\cite{[H3]}
Let $S$ be a set of positive integers, and let $x\in S$ with $g(x)\ge2$.
(i) We say that two distinct greatest-type divisors $y_1$ and $y_2$ of $x$ in $S$ satisfy condition $\mathcal{G}$ if $[y_1,y_2]=x$ and $(y_1,y_2)\in G_S(y_1)\cap G_S(y_2)$.
(ii) We say that $x$ satisfies condition $\mathcal{G}$ if any two distinct greatest-type divisors of $x$ in $S$ satisfy condition $\mathcal{G}$.
\end{defn}

It should be pointed out that if $S$ is gcd closed, then the definition of condition $\mathcal{G}$ given above coincides with that of
 condition $\mathcal{C}$ given in \cite{[FHZ]}.

\begin{defn}\cite{[H3]}
Let $S$ be a set of positive integers. We say that the set $S$ satisfies condition $\mathcal{G}$ if every element $x\in S$ satisfies
 either $g(x)\le1$, or $g(x)\ge2$ and $x$ satisfies condition $\mathcal{G}$.
\end{defn}

Now we give an important property of a gcd closed set $S$ satisfying condition $\mathcal{G}$.

First, by the definition of gcd closed, we find that for $S=\{x_1,x_2,...,x_n\}$, we have $(x_1,x_2,...,x_n) \in S$, and $\{\frac{x_1}{(x_1,x_2,...,x_n)},\frac{x_2}{(x_1,x_2,...,x_n)},...,\frac{x_n}{(x_1,x_2,...,x_n)}\}$ is also gcd closed.
 Therefore, without loss of generality, we always assume that $(x_1,x_2,...,x_n)=1$, and the gcd closed set $S$ satisfies condition $\mathcal{G}$.

\begin{pro}\label{pro4.1}
For every $x\in S$ with $g(x)=k\ge 2$, let $G_S(x)=\{y_1,...,y_k\}$. Then we have $x=(y_1,...,y_k)\prod_{i=1}^k\frac{x}{y_i}$.
\end{pro}

\begin{proof}
In fact, for all $i\ne j$, we have $\frac{y_iy_j}{(y_i,y_j)}=x$, which implies $\frac{y_j}{(y_i,y_j)}=\frac{x}{y_i}$ and $\frac{y_i}{(y_i,y_j)}=\frac{x}{y_j}$. Therefore, $(\frac{x}{y_j},\frac{x}{y_i})=(\frac{y_i}{(y_i,y_j)},\frac{y_j}{(y_i,y_j)})=1$.

Moreover, $(y_i,y_j)\in G_S(y_i)\cap G_S(y_j)$. By induction,
$$G_S((y_{i_1},y_{i_2},...,y_{i_j}))\supseteq \{(y_{i_1},y_{i_2},...,y_{i_j},y_l)\mid l\not=i_1,...,i_j\}.$$

Now we compute
$$\prod_{i=1}^k\frac{x}{y_i}=\frac{x}{y_1}\prod_{i=2}^k\frac{y_1}{(y_1,y_i)}=\frac{x}{y_1}\frac{y_1}{(y_1,y_2)}\prod_{i=3}^k\frac{(y_1,y_2)}{(y_1,y_2,y_i)}$$
$$\cdots =\frac{x}{y_1}\frac{y_1}{(y_1,y_2)}\prod_{i=3}^k\frac{(y_1,y_2,...,y_{i-1})}{(y_1,y_2,...,y_i)}=\frac{x}{(y_1,...,y_k)},$$
which completes the proof.
\end{proof}

\begin{cor}\label{cor4.1}
Let $x$ and $G_S(x)$ be as above. Then for $j<k$,
$$G_S((y_{i_1},y_{i_2},...,y_{i_j}))=\{(y_{i_1},y_{i_2},...,y_{i_j},y_l)\mid l\not=i_1,...,i_j\}.$$
\end{cor}

\begin{proof}
Since
$$y_1=(y_1,...,y_k)\prod_{i=2}^k\frac{x}{y_i}=(y_1,...,y_k)\prod_{i=2}^k\frac{y_1}{(y_1,y_i)}=((y_1,y_2),...,(y_1,y_k))\prod_{i=2}^k\frac{y_1}{(y_1,y_i)},$$
the corollary holds for $y_1$. Similarly, it holds for all $y_i$. The general case follows by induction on $j$ and $(y_{i_1},y_{i_2},...,y_{i_j})$.
\end{proof}

Now we define the imitative M\H{o}bius function of $x$ as follows:
$$
\mu_x(t)=
\begin{cases}
1, & t=1,\\
(-1)^j, & t=\prod_{k=1}^j\frac{x}{y_{i_k}},\\
0, & \text{otherwise}.
\end{cases}
$$

By the definition, if $y\in G_S(x)$, then
$$\mu_y(t)=\mu_x(t),\quad \text{for }t\mid\prod_{y'\in G_S(x),y'\not=y}\frac{x}{y'}.$$

\begin{lem}\label{lem4.1}
For $h\in S$, we have
$$
\sum_{d|x,d\in S}\mu_x(\frac{x}{d})(d,h)^a=
\begin{cases}
(y_1,...,y_k)^a\prod_{i=1}^k((\frac{x}{y_i})^a-1), & \text{if } x|h,\\
0, & \text{otherwise}.
\end{cases}
$$
\end{lem}

\begin{proof}
If $x\nmid h$, then there exists $y_i$ such that $(x,h)|y_i$, because $(x,h)|x$ but $(x,h)\ne x$.

Assume that $(x,h)|y_1$. Then for every $d=(y_1,...,y_k)\prod_{j=1}^l\frac{x}{y_{i_j}}$ with $i_1,...,i_l\ne 1$, we have $(d,h)=(\frac{x}{y_1}d,h)$, since
$(\frac{x}{y_1}d,h)=(y_1,\frac{x}{y_1}d,h)=((y_1,\frac{x}{y_1}d),h)=(d,h)$.

Therefore, in the sum $\sum_{d|x,d\in S}\mu_x(\frac{x}{d})(d,h)^a$, we can pair $d$ with $\frac{x}{y_1}d$, and we obtain
$$\mu_x(\frac{x}{d})(d,h)^a+\mu_x(\frac{y_1}{d})(\frac{x}{y_1}d,h)^a=\mu_x(\frac{y_1}{d})\left(\mu_x(\frac{x}{y_1})(d,h)^a+(\frac{x}{y_1}d,h)^a\right)=0.$$

Thus,
$$\sum_{d|x,d\in S}\mu_x(\frac{x}{d})(d,h)^a=\sum_{d|y_1,d\in S}\left(\mu_x(\frac{x}{d})(d,h)^a+\mu_x(\frac{y_1}{d})(\frac{x}{y_1}d,h)^a\right)=0.$$

If $x|h$, then $(d,h)=d$, and hence
\begin{align*}
\sum_{d|x,d\in S}\mu_x(\frac{x}{d})(d,h)^a&=\sum_{d|x,d\in S}\mu_x(\frac{x}{d})d^a\\
&=x^a+\sum_{i=1}^k\mu_x(\frac{x}{y_i})y_i^a+...+\mu_x(\prod_{i=1}^k\frac{x}{y_i})(y_1,...,y_k)^a\\
&=(y_1,...,y_k)^a\prod_{i=1}^k\left((\frac{x}{y_i})^a+\mu_x(\frac{x}{y_i})\right)=(y_1,...,y_k)^a\prod_{i=1}^k((\frac{x}{y_i})^a-1).
\end{align*}

This completes the proof of the lemma.
\end{proof}

\begin{defn}\label{def4.3}
For $S=\{x_1,...,x_n\}$ and a positive integer $a$, we define the matrix $B_S$ whose $(i,j)$-entry is
$$
\begin{cases}
\mu_{x_j}(\frac{x_j}{x_i}), & \text{if } x_i|x_j,\\
0, & \text{otherwise},
\end{cases}
$$
and we define the matrix $[B_{S^a}]$ whose $(i,j)$-entry is
$$
\begin{cases}
\mu_{x_j}(\frac{x_j}{x_i})(\frac{x_j}{x_i})^a, & \text{if } x_i|x_j,\\
0, & \text{otherwise}.
\end{cases}
$$
\end{defn}

\begin{lem}\label{lem4.2}
The $(i,j)$-entry of $B_S^{-1}$ is
$$
\begin{cases}
1, & \text{if } x_i|x_j,\\
0, & \text{otherwise},
\end{cases}
$$
and the $(i,j)$-entry of $[B_{S^a}]^{-1}$ is
$$
\begin{cases}
(\frac{x_j}{x_i})^a, & \text{if } x_i|x_j,\\
0, & \text{otherwise}.
\end{cases}
$$
\end{lem}

\begin{proof}
We compute the $(i,j)$-entry of $B_S^{-1}B_S$:
$$\sum_{x_i|x_k|x_j}\mu_{x_j}(\frac{x_j}{x_k}).$$

Assume that all integers that divide $x_j$ and are greatest-type divisors of $x_j$ and are multiples of $x_i$ are $y_1,...,y_l$. Then
$$\sum_{x_i|x_k|x_j}\mu_{x_j}(\frac{x_j}{x_k})=\prod_{k=1}^l(1+\mu_{x_j}(\frac{x_j}{y_k}))=0,$$
where the last equality holds when $x_j\ne x_i$.

For the $(i,j)$-entry of $[B_{S^a}]^{-1}[B_{S^a}]$,
$$\sum_{x_i|x_k|x_j}\mu_{x_j}(\frac{x_j}{x_k})(\frac{x_j}{x_k})^a(\frac{x_k}{x_i})^a
=(\frac{x_j}{x_i})^a\sum_{x_i|x_k|x_j}\mu_{x_j}(\frac{x_j}{x_k}),$$
and the remaining proof is similar to that for $B_S^{-1}B_S$.
\end{proof}

Now we can consider the divisibility properties of $((x_i,x_j)^a)$ and $([x_i,x_j]^a)$.

\begin{thm}\label{thm4.1}
Let $S=\{x_1,...,x_n\}$ be a gcd closed set satisfying condition $\mathcal{G}$, and let $a$ be a positive integer. For each $x_i$, let $G_S(x_i)=\{y^i_1,...,y^i_{l_i}\}$. Then
\begin{align*}
B_S^T((x_i,x_j)^a)B_S&=diag(1,(y^2_1,...,y^2_{l_2})^a\prod_{h=1}^{l_2}((\frac{x_2}{y^2_h})^a-1),...,(y^n_1,...,y^n_{l_n})^a\prod_{h=1}^{l_n}((\frac{x_n}{y^n_h})^a-1)),\\
[B_{S^a}]^T([x_i,x_j]^a)[B_{S^a}]&=diag(1,x_2^a\prod_{h=1}^{l_2}(1-(\frac{x_2}{y^2_h})^a),...,x_n^a\prod_{h=1}^{l_n}(1-(\frac{x_n}{y^n_h})^a)).
\end{align*}
\end{thm}

\begin{proof}
Similar to the proof of Theorem \ref{thm1.2} and Corollary \ref{cor3.1}, we replace $p|x_i$ with $\frac{x_i}{y^i_j}$.
\end{proof}

Now we can prove the following result:

\begin{thm}\label{thm4.2}
Let $S=\{x_1,...,x_n\}$ be a gcd closed set satisfying condition $\mathcal{G}$, and let $a,b$ be positive integers with $a\mid b$. Then
$$((x_i,x_j)^a)|((x_i,x_j)^b),\quad ((x_i,x_j)^a)|([x_i,x_j])^a,\quad \text{and}\quad ([x_i,x_j])^a|([x_i,x_j])^b.$$
\end{thm}

\begin{proof}
We begin by computing
\begin{align*}
((x_i,x_j)^a)^{-1}((x_i,x_j)^b)=& B_Sdiag(1,(y^2_1,...,y^2_{l_2})^{b-a}\prod_{h=1}^{l_2}\frac{((\frac{x_2}{y^2_h})^b-1)}{((\frac{x_2}{y^2_h})^a-1)},...,\\
&(y^n_1,...,y^n_{l_n})^{b-a}\prod_{h=1}^{l_n}\frac{((\frac{x_n}{y^n_h})^b-1)}{((\frac{x_n}{y^n_h})^a-1)})B_S^{-1}.
\end{align*}

Next,
\begin{align*}
((x_i,x_j)^a)^{-1}([x_i,x_j])^a=& B_Sdiag(1,(y^2_1,...,y^2_{l_2})^{-a}\prod_{h=1}^{l_2}\frac{1}{((\frac{x_2}{y^2_h})^a-1)},...,\\
&(y^n_1,...,y^n_{l_n})^{-a}\prod_{h=1}^{l_n}\frac{1}{((\frac{x_n}{y^n_h})^a-1)})B_S^T\\
&([B_{S^a}]^T)^{-1}diag(1,x_2^a\prod_{h=1}^{l_2}(1-(\frac{x_2}{y^2_h})^a),...,x_n^a\prod_{h=1}^{l_n}(1-(\frac{x_n}{y^n_h})^a))[B_{S^a}]^{-1}.
\end{align*}

The $(i,j)$-entry of $B_S^T([B_{S^a}]^T)^{-1}$ is
$$\sum_{x_j|x_k|x_i}\mu_{x_i}(\frac{x_i}{x_k})(\frac{x_k}{x_j})^a=(x_j)^{-a}(y^i_{l'_1},...,y^i_{l'_{m_{ij}}})^a\prod_{h=1}^{m_{ij}}((\frac{x_i}{y^i_{l'_h}})^a-1),$$
where $x_j|(y^i_{l'_1},...,y^i_{l'_{m_{ij}}})$, but for every $k\not\in \{l'_1,...,l'_{m_{ij}}\}$, we have $x_j\nmid y^i_k$ and $\{l'_1,...,l'_{m_{ij}}\}\subseteq \{1,...,l_i\}$.

Thus the $(i,j)$-entry of
\begin{align*}
&diag(1,(y^2_1,...,y^2_{l_2})^{-a}\prod_{h=1}^{l_2}\frac{1}{((\frac{x_2}{y^2_h})^a-1)},...,(y^n_1,...,y^n_{l_n})^{-a}\prod_{h=1}^{l_n}\frac{1}{((\frac{x_n}{y^n_h})^a-1)})B_S^T\\
&([B_{S^a}]^T)^{-1}diag(1,x_2^a\prod_{h=1}^{l_2}(1-(\frac{x_2}{y^2_h})^a),...,x_n^a\prod_{h=1}^{l_n}(1-(\frac{x_n}{y^n_h})^a))
\end{align*}
is
$$\frac{\prod_{h=1}^{l_j}(1-(\frac{x_j}{y^j_h})^a)(y^i_{l'_1},...,y^i_{l'_{m_{ij}}})^a\prod_{h'=1}^{m_{ij}}((\frac{x_i}{y^i_{l'_{h'}}})^a-1)}{(y^i_1,...,y^i_{l_i})^a\prod_{h''=1}^{l_i}((\frac{x_i}{y^i_{h''}})^a-1)}.$$

Since $\{y^i_{l'_1},...,y^i_{l'_{m_{ij}}}\}\subseteq G_S(x_i)$, we have $(y^i_1,...,y^i_{l_i})\mid(y^i_{l'_1},...,y^i_{l'_{m_{ij}}})$.

Consider $y\in G_S(x_i)-\{y^i_{l'_1},...,y^i_{l'_{m_{ij}}}\}$. Because $x_j|x_i$ and $x_j\nmid y$, we have $y|[x_j,y]|x_i$ and $y\ne [x_j,y]$, so $[x_j,y]=x_i$.

Therefore, $\frac{x_jy}{(x_j,y)}=x_i$, which implies $\frac{x_j}{(x_j,y)}=\frac{x_i}{y}$.

Next we prove that $(x_j,y)$ is a greatest-type divisor of $x_j$.

If not, since $x_j\nmid y$, we have $(x_j,y)\ne x_j$. Then there exists $y'\in G_S(x_j)$ such that $(x_j,y)|y'$ and $(x_j,y)\ne y'$. But $(y',y)|(x_j,y)$, so $[y',y]=\frac{y'y}{(y',y)}>y$, and hence $[y',y]=x_i$. Then $\frac{x_j}{(x_j,y)}=\frac{y'}{(y',y)}$, so $x_j=\frac{y'(x_j,y)}{(y',y)}$, i.e., $x_j=[y',(x_j,y)]=y'$, a contradiction. Thus $y'=(x_j,y)$.

Consequently, $\frac{x_j}{y'}=\frac{x_i}{y}$. Hence,
$$\prod_{y\in G_S(x_i)-\{y^i_{l'_1},...,y^i_{l'_{m_{ij}}}\}}((\frac{x_i}{y})^a-1)\mid\prod_{h=1}^{l_j}(1-(\frac{x_j}{y^j_h})^a),$$
that is,
\[\prod_{h''=1}^{l_i}((\frac{x_i}{y^i_{h''}})^a-1)\mid\prod_{h=1}^{l_j}(1-(\frac{x_j}{y^j_h})^a)\prod_{h'=1}^{m_{ij}}((\frac{x_i}{y^i_{l'_{h'}}})^a-1).\tag{4.1}\]

Therefore,
\begin{align*}
&diag(1,(y^2_1,...,y^2_{l_2})^{-a}\prod_{h=1}^{l_2}\frac{1}{((\frac{x_2}{y^2_h})^a-1)},...,(y^n_1,...,y^n_{l_n})^{-a}\prod_{h=1}^{l_n}\frac{1}{((\frac{x_n}{y^n_h})^a-1)})B_S^T\\
&([B_{S^a}]^T)^{-1}diag\{1,x_2^a\prod_{h=1}^{l_2}(1-(\frac{x_2}{y^2_h})^a),...,x_n^a\prod_{h=1}^{l_n}(1-(\frac{x_n}{y^n_h})^a)\}\in M_n(\mathbb{Z}).
\end{align*}

Finally, we compute
\begin{align*}
([x_i,x_j]^a)^{-1}([x_i,x_j]^b)=& [B_{S^a}]diag\{1,\frac{1}{x_2^a\prod_{h=1}^{l_2}(1-(\frac{x_2}{y^2_h})^a)},...,\frac{1}{x_n^a\prod_{h=1}^{l_n}(1-(\frac{x_n}{y^n_h})^a)}\}[B_{S^a}]^T\\
&([B_{S^b}]^T)^{-1}diag\{1,x_2^b\prod_{h=1}^{l_2}(1-(\frac{x_2}{y^2_h})^b),...,x_n^b\prod_{h=1}^{l_n}(1-(\frac{x_n}{y^n_h})^b)\}[B_{S^b}]^{-1}.
\end{align*}

The $(i,j)$-entry of $[B_{S^a}]^T([B_{S^b}]^T)^{-1}$ is:
\begin{align*}
&\sum_{x_j|x_k|x_i}\mu_{x_i}(\frac{x_i}{x_k})(\frac{x_i}{x_k})^a(\frac{x_k}{x_j})^b=\frac{x_i^a}{x_j^b}\sum_{x_j|x_k|x_i}\mu_{x_i}(\frac{x_i}{x_k})x_k^{b-a}\\
=&\frac{x_i^a}{x_j^b}(y_{l'_1}^i,...,y_{l'_{m_{ij}}}^i)^{b-a}\prod_{h=1}^{m_{ij}}((\frac{x_i}{y_{l'_h}^i})^{b-a}-1).
\end{align*}

Hence the $(i,j)$-entry of
\begin{align*}
&diag(1,\frac{1}{x_2^a\prod_{h=1}^{l_2}(1-(\frac{x_2}{y^2_h})^a)},...,\frac{1}{x_n^a\prod_{h=1}^{l_n}(1-(\frac{x_n}{y^n_h})^a)})[B_{S^a}]^T\\
&([B_{S^b}]^T)^{-1}diag(1,x_2^b\prod_{h=1}^{l_2}(1-(\frac{x_2}{y^2_h})^b),...,x_n^b\prod_{h=1}^{l_n}(1-(\frac{x_n}{y^n_h})^b))
\end{align*}
is
$$\frac{\prod_{h=1}^{l_j}(1-(\frac{x_j}{y^j_h})^b)(y_{l'_1}^i,...,y_{l'_{m_{ij}}}^i)^{b-a}\prod_{h'=1}^{m_{ij}}((\frac{x_i}{y_{l'_{h'}}}^i)^{b-a}-1)}{\prod_{{h''}=1}^{l_i}(1-(\frac{x_i}{y^i_{h''}})^a)}.$$

By the same reasoning as (4.1), we have
$$\prod_{h=1}^{l_i}(1-(\frac{x_i}{y^i_h})^a)\mid\prod_{h=1}^{l_j}(1-(\frac{x_j}{y^j_h})^b)\prod_{h'=1}^{m_{ij}}((\frac{x_i}{y_{l'_{h'}}}^i)^{b-a}-1).$$

Therefore,
\begin{align*}
&diag(1,\frac{1}{x_2^a\prod_{h=1}^{l_2}(1-(\frac{x_2}{y^2_h})^a)},...,\frac{1}{x_n^a\prod_{h=1}^{l_n}(1-(\frac{x_n}{y^n_h})^a)})B_a^T\\
&(B_b^T)^{-1}diag(1,x_2^b\prod_{h=1}^{l_2}(1-(\frac{x_2}{y^2_h})^b),...,x_n^b\prod_{h=1}^{l_n}(1-(\frac{x_n}{y^n_h})^b))\in M_n(\mathbb{Z}).
\end{align*}

This completes the proof of Theorem \ref{thm4.2}.
\end{proof}

\bibliographystyle{amsplain}

\end{document}